\documentclass[11pt]{amsart}

\usepackage{tikz}
\usepackage{dbnsymb} 

\usetikzlibrary{knots}
\usetikzlibrary{decorations.markings}

\usepackage{graphicx} 
\usepackage{t1enc}
\usepackage{amsthm,amssymb}
\usepackage{latexsym}
\usepackage{amsmath} 
\usepackage{graphics}
\usepackage{epsfig,multicol}
\usepackage{amsfonts}
\usepackage[latin1]{inputenc}
\usepackage[english]{babel}
\usepackage{ mathrsfs }

\newtheorem*{theorem*}{Theorem}
\newtheorem*{proposition*}{Proposition}
\newtheorem*{lemma*}{Lemma}

{\theoremstyle{definition} }

\def\hpic #1 #2 {\mbox{$\begin{array}[c]{l} \epsfig{file=#1,height=#2}
\end{array}$}}
 
\def\vpic #1 #2 {\mbox{$\begin{array}[c]{l} \epsfig{file=#1,width=#2}
\end{array}$}}

\newcommand  {\rmn}\romannumeral

\begin{document}
\title[On the Alexander theorem for the oriented Thompson group $\vec F$]{On the Alexander theorem for the oriented Thompson group $\vec F$}
\author{Valeriano Aiello} 
\address{Valeriano Aiello,
Section de Math\'  ematiques, Universit\' e de Gen\` eve, 2-4 rue du Li\` evre, Case Postale 64,
1211 Gen\` eve 4, Switzerland
}\email{valerianoaiello@gmail.com}

\begin{abstract}
In \cite{Jo14} and \cite{Jo18} Vaughan Jones introduced a construction which yields oriented knots and links from elements of the oriented Thompson group $\vec{F}$.
In this paper we prove, by analogy with Alexander's
classical theorem establishing that every knot or link can be represented as a closed braid, that given an oriented knot/link $\vec{L}$, there exists an element $g$ in $\vec{F}$ whose \emph{closure}   $\vec{\mathcal{L}}(g)$ is $\vec L$.   
\end{abstract}
\maketitle
 
\section{Introduction}
Recently, Vaughan Jones has introduced a method to construct unitary representations of the Thompson groups (and more in general of groups of fractions of certain categories) \cite{Jo14, Jo16}, many of which arise from Planar Algebras \cite{jo2}.
A connection between un-oriented knots/links and the Thompson groups $F$ and $T$ was also established 
by showing that, given any element of these groups, one may construct a knot/link
(see \cite{Jo18} for a review of the recent development in this field, and also \cite{ACJ, AiCo1, AiCo2}). 
Moreover, it was shown that the elements of the oriented subgroups $\vec{F}$ and $\vec{T}$ give rise to oriented knots/links. 
Within this framework Jones proved   a result analogous to the Alexander Theorem for   braids and links \cite{Alex}. 
In particular, he showed that given any unoriented knot/link $L$ there exists an element $g$ in $F$ whose associated link is  equal to $L$. For the oriented case a similar (but slightly weaker) result was proven. In fact, the knot/link was shown to be reproduced only up to disjoint union with unlinked unknots 
\cite[Theorem 5.3.15]{Jo14}. In this paper we show that  actually the oriented knot/link can be exactly reproduced. This answers a question posed by Jones in \cite[Section 7, Question (2)]{Jo18} and reinforces the idea that the oriented subgroup $\vec{F}$ can be considered as an alternative to the braid groups. 

The paper is organised in two sections. In the first one we recall the definitions of the Thompson group $F$ and of the oriented subgroup $\vec{F}$, we then describe the construction of oriented knots/links from elements of $\vec F$ providing an example. In the second section, after recalling the definition of the semi-dual graph, we state and prove the main result of this paper.

\section{Preliminaries on the Thompson groups and Jones' construction of knots and links}
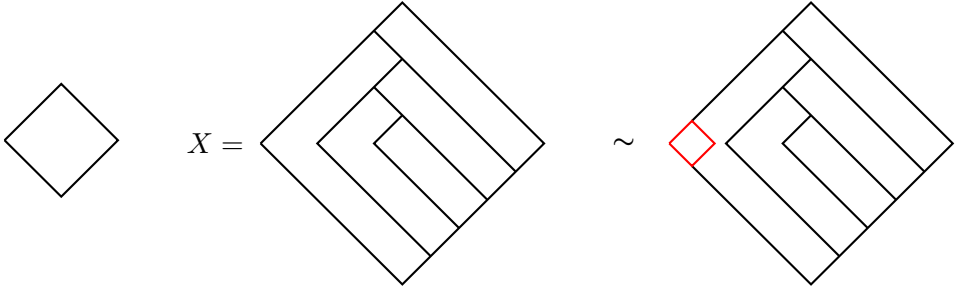
\begin{figure}
\caption{A pair of opposing carets and two equivalent elements in $F$.}\label{AAA}
\phantom{This text will be invisible} 
\[\begin{tikzpicture}[x=1.5cm, y=1.5cm,
    every edge/.style={
        draw,
      postaction={decorate,
                    decoration={markings}
                   }
        }
]

\draw[thick] (0,0)--(.5,.5)--(1,0)--(.5,-.5)--(0,0);
\node at (0,-1.2) {$\;$};
\end{tikzpicture}
\qquad 
\begin{tikzpicture}[x=1.5cm, y=1.5cm,
    every edge/.style={
        draw,
      postaction={decorate,
                    decoration={markings}
                   }
        }
]

\node at (-.4,0) {$\scalebox{1}{$X=$}$};

\draw[thick]  (0,0)--(1.25,1.25)--(2.5,0)--(1.25,-1.25)--(0,0);
\draw[thick]  (1,1)--(2.25,.-.25);
\draw[thick]  (1,.5)--(2,-.5);
\draw[thick]  (1.25,.75)--(.5,0)--(1.5,-1);
\draw[thick]  (1.75,-.75)--(1,0)--(1.25,.25);


\end{tikzpicture}
\qquad 
\begin{tikzpicture}[x=1.5cm, y=1.5cm,
    every edge/.style={
        draw,
      postaction={decorate,
                    decoration={markings}
                   }
        }
]

\node at (-.4,0) {$\scalebox{1}{$\thicksim$}$};

\draw[thick]  (0.2,0.2)--(1.25,1.25)--(2.5,0)--(1.25,-1.25)--(0.2,-0.2);
\draw[thick]  (1,1)--(2.25,.-.25);
\draw[thick]  (1,.5)--(2,-.5);
\draw[thick]  (1.25,.75)--(.5,0)--(1.5,-1);
\draw[thick]  (1.75,-.75)--(1,0)--(1.25,.25);
\draw[thick, red] (0,0)--(.2,.2)--(.4,0)--(.2,-.2)--(0,0);

\end{tikzpicture}
\]
\end{figure}
There are many   equivalent descriptions of R. Thompson's group $F$. The one 
 that is most appropriate for our work in this
paper uses tree diagrams. 
An element of $F$ is given by a pair of rooted, planar, binary trees $(T_+,T_-)$ with the same number of leaves. As usual, we draw a pair of trees in the plane with one tree upside down on top of the other.
Two pairs of trees are equivalent if they differ by a pair of opposing carets, see the Figure \ref{AAA}. Thanks to this equivalence relation, the following rule defines the multiplication in $F$: $(T_+,T)\cdot (T,T_-):=(T_+,T_-)$. The trivial element is represented by any pair $(T,T)$ and the inverse of $(T_+,T_-)$ is just $(T_-,T_+)$. 
We refer to \cite{CFP} and \cite{B} for further details and more information on the Thompson group $F$.

Given an element $(T_+,T_-)$ of the Thompson group $F$ one may construct the graph $\Gamma(T_+,T_-)$,  \cite[Section 4.1]{Jo14}. We describe the procedure with an example. We pick the element $X$ in Figure \ref{AAA}. 
We place the leaves of $(T_+,T_-)$ on the half-integers, 
i.e. they are $(1/2,0)$, $(3/2,0)$, $(5/2,0)$, $(7/2,0)$, $(9/2,0)$, $(11/2,0)$. The vertices of $\Gamma(T_+,T_-)$  are $(0,0)$, $(1,0)$, $(2,0)$, $(3,0)$, $(4,0)$, $(5,0)$. 
Now the edges of $\Gamma(T_+,T_-)$ pass transversally through the edges of the top tree sloping up from left to right (herein we refer to them by the West-North edges, or just WN$=$\rotatebox[origin=tr]{-45}{|})
and the edges of the bottom tree sloping down from left to right (we call them the West-South edges or simply  WS$=$\rotatebox[origin=tr]{45}{|}).
See Figure \ref{B2} for the  $\Gamma$-graph of the element $X$ depicted in Figure \ref{AAA}.  \\
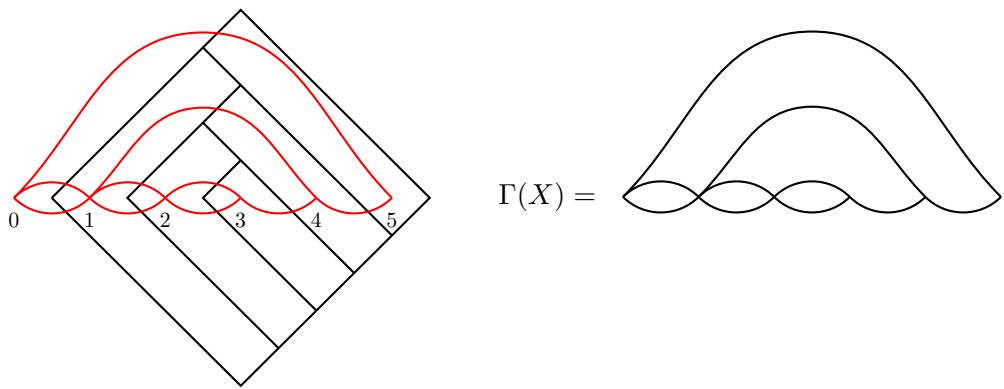
\begin{figure}
\caption{An example of a $\Gamma$-graph.}\label{B2}
\[\begin{tikzpicture}[x=2cm, y=2cm,
    every edge/.style={
        draw,
      postaction={decorate,
                    decoration={markings}
                   }
        }
]

\node at (-.25,-.15) {$\scalebox{.75}{$0$}$};
\node at (.25,-.15) {$\scalebox{.75}{$1$}$};
\node at (.75,-.15) {$\scalebox{.75}{$2$}$};
\node at (1.25,-.15) {$\scalebox{.75}{$3$}$};
\node at (1.75,-.15) {$\scalebox{.75}{$4$}$};
\node at (2.25,-.15) {$\scalebox{.75}{$5$}$};

\draw[thick]  (0,0)--(1.25,1.25)--(2.5,0)--(1.25,-1.25)--(0,0);
\draw[thick]  (1,1)--(2.25,.-.25);
\draw[thick]  (1,.5)--(2,-.5);
\draw[thick]  (1.25,.75)--(.5,0)--(1.5,-1);
\draw[thick]  (1.75,-.75)--(1,0)--(1.25,.25);

\draw[thick, red] (-.25,0) to[out=45,in=135] (.25,0);

\draw[thick, red] (-.25,0) to[out=45,in=180] (1,1.1);
\draw[thick, red] (1,1.1) to[out=0,in=135] (2.25,0);

\draw[thick, red] (.25,0) to[out=45,in=135] (.75,0);
\draw[thick, red] (.25,0) to[out=45,in=180] (1,.6);
\draw[thick, red] (1,.6) to[out=0,in=135] (1.75,0);

\draw[thick, red] (.75,0) to[out=45,in=135] (1.25,0);

\draw[thick, red] (-.25,0) to[out=-45,in=-135] (.25,0);
\draw[thick, red] (.25,0) to[out=-45,in=-135] (.75,0);
\draw[thick, red] (.75,0) to[out=-45,in=-135] (1.25,0);
\draw[thick, red] (1.25,0) to[out=-45,in=-135] (1.75,0);
\draw[thick, red] (1.75,0) to[out=-45,in=-135] (2.25,0);

\end{tikzpicture}
\qquad
\begin{tikzpicture}[x=2cm, y=2cm,
    every edge/.style={
        draw,
      postaction={decorate,
                    decoration={markings}
                   }
        }
]


\node at (-.75,0) {$\scalebox{1}{$\Gamma(X)=$}$};

\draw[thick] (-.25,0) to[out=45,in=135] (.25,0);

\draw[thick] (-.25,0) to[out=45,in=180] (1,1.1);
\draw[thick] (1,1.1) to[out=0,in=135] (2.25,0);

\draw[thick] (.25,0) to[out=45,in=135] (.75,0);
\draw[thick] (.25,0) to[out=45,in=180] (1,.6);
\draw[thick] (1,.6) to[out=0,in=135] (1.75,0);

\draw[thick] (.75,0) to[out=45,in=135] (1.25,0);

\draw[thick] (-.25,0) to[out=-45,in=-135] (.25,0);
\draw[thick] (.25,0) to[out=-45,in=-135] (.75,0);
\draw[thick] (.75,0) to[out=-45,in=-135] (1.25,0);
\draw[thick] (1.25,0) to[out=-45,in=-135] (1.75,0);
\draw[thick] (1.75,0) to[out=-45,in=-135] (2.25,0);

\node at (0,-1.2) {$\;$};

\end{tikzpicture}
\]
\end{figure}

%

We observe that $\Gamma(T_+,T_-)$  comes with a natural orientation \cite[Section 4.1]{Jo14}. Indeed, if an edge has vertices $(a,0)$ and $(b,0)$ with $a<b$, we call $(a,0)$ the source and $(b,0)$ the target. The edges may also be endowed with signs: $+$ for the edges in the upper-half plane and $-$ for the edges in the lower-half plane.
We point out that giving $\Gamma(T_+,T_-)$  is equivalent to giving the pair of trees $(T_+,T_-)$ as the latter may be reconstructed (for a proof see \cite[Lemma 4.1.4]{Jo14}).

We are 
now in a position to define the oriented subgroup $\vec{F}$ 
$$
\vec{F}:=\{(T_+,T_-)\in F \textrm{ s.t. $\Gamma(T_+,T_-)$ is $2$-colourable} \} \; ,
$$
where by being $2$-colourable we mean that it is possible to label the vertices of the graph with two colours such that whenever two vertices are connected by an edge, they have different colours. See \cite[Chapter V]{Bollobas} for more information about vertex colourings.
We denote the two colours by $+$ and $-$. Since the graph $\Gamma(T_+,T_-)$ is connected, if it is $2$-colourable, there are exactly two colourings. By convention we choose that in which the leftmost vertex has colour $+$.
We mention that this group was shown to be isomorphic to the Brown-Thompson group $F_3$   (\cite[Lemma 4.7]{Golan}, see also \cite[Theorem 5.5]{Ren} for an alternative proof).

In Jones' construction of knots and links, after taking an element of the Thompson group and producing the $\Gamma$-graph, we need  to make  three further steps. The first one is the construction of the medial graph $M(\Gamma(T_+,T_-))$ of $\Gamma(T_+,T_-)$. 
We recall from \cite[Chapter X]{Bollobas} that given a connected plane graph $G$, its medial graph $M(G)$ is obtained by putting a vertex on every edge of $G$ and by joining   two new vertices by an edge lying in a face of $G$ if the vertices are on adjacent edges of the face.
Since all the vertices of $M(\Gamma(T_+,T_-))$ are $4$-valent, in order to obtain a knot/link we just need a rule to turn the vertices into crossings. For the vertices in upper-half plane we use the crossing $\slashoverback$,
whereas for those in lower-half plane we use 
$\backoverslash$.
%
 For example, for the element $X$ given in Figure \ref{AAA} we get the graph and link diagram shown in Figure \ref{fig4} (a). \\
%
Now it is possible to shade the the link diagram in black and white (by convention the colour of the unbounded region is white) and obtain a surface in $\mathbb{R}^3$ whose boundary is the knot/link $\mathcal{L}(T_+,T_-)$, see \cite[Section 5.3.2]{Jo14}. By construction the vertices of the graph $\Gamma(T_+, T_-)$ are precisely in the black regions and each one has been assigned with a colour $+$ or $-$. One can easily see that, in our example, the colours of the vertices of $\Gamma(T_+,T_-)$ are $+$, $-$, $+$, $-$, $+$, $-$ (from left to right). These colours determine an orientation of the surface and of the boundary ($+$ means that the region is positively oriented). 
Therefore, we get the oriented link shown in Figure \ref{fig4} (b).
\begin{figure}
\caption{From the medial graph of a $\Gamma$-graph to an oriented link. }\label{fig4}
\phantom{This text will be invisible} 
\[
\phantom{This text} 
\begin{tikzpicture}[x=.6cm, y=.6cm, every path/.style={
 thick}, 
every node/.style={transform shape, knot crossing, inner sep=1.5pt}]
\node (aaaa) at (-4.75,1) {$\scalebox{1}{$(a)\;  M(\Gamma(T_+,T_-))=$}$}; 
%

\node (a1) at (0,0) {};
\node (a2) at (1,0) {};
\node (a3) at (2,0) {};
\node (a4) at (3,0) {};
\node (a5) at (4,0) {};

\node (b1) at (0,1) {};
\node (b2) at (1,1) {};
\node (b3) at (2,1) {};
\node (b4) at (3,2) {};
\node (b5) at (4,3) {};

\node (appo) at (0,3) {};

\draw (a1.center) .. controls (a1.4 north west) and (b1.4 south west) .. (b1.center);
\draw (a1.center) .. controls (a1.4 north east) and (b1.4 south east) .. (b1.center);

\draw (a2.center) .. controls (a2.4 north west) and (b2.4 south west) .. (b2.center);
\draw (a2.center) .. controls (a2.4 north east) and (b2.4 south east) .. (b2.center);

\draw (b1.center) .. controls (b1.8 north east) and (b4.8 north west) .. (b4.center);
\draw (a1.center) .. controls (a1.4 south east) and (a2.4 south west) .. (a2.center);

\draw (b1.center) .. controls (b1.16 north west) and (b5.4 south west) .. (b5.center);
\draw (b2.center) .. controls (b2.8 north west) and (b4.4 south west) .. (b4.center);

\draw (b2.center) .. controls (b2.4 north east) and (b3.4 north west) .. (b3.center);
\draw (a2.center) .. controls (a2.4 south east) and (a3.4 south west) .. (a3.center);
\draw (a3.center) .. controls (a3.4 north east) and (b3.4 south east) .. (b3.center);
\draw (a3.center) .. controls (a3.4 north west) and (b3.4 south west) .. (b3.center);
\draw (a1.center) .. controls (a1.16 south west) and (appo.16 west) .. (appo.center);
\draw (appo.center) .. controls (appo.4 east) and (b5.4 north west) .. (b5.center);

\draw (a3.center) .. controls (a3.4 south east) and (a4.4 south west) .. (a4.center);
\draw (a4.center) .. controls (a4.4 south east) and (a5.4 south west) .. (a5.center);
\draw (a4.center) .. controls (a4.4 north east) and (b4.4 south east) .. (b4.center);
\draw (b4.center) .. controls (b4.8 north east) and (a5.4 north west) .. (a5.center);
\draw (a5.center) .. controls (a5.4 north east) and (b5.4 south east) .. (b5.center);

\draw (a5.center) .. controls (a5.16 south east) and (b5.16 north east) .. (b5.center);
\draw (b3.center) .. controls (b3.4 north east) and (a4.4 north west) .. (a4.center);

\end{tikzpicture}
\qquad\qquad\qquad\qquad 
\begin{tikzpicture}[x=.6cm, y=.6cm, every path/.style={
 thick}, 
every node/.style={transform shape, knot crossing, inner sep=1.5pt}]

\node (aaaa) at (-4.5,1) {$\scalebox{1}{$\mathcal{L}(T_+,T_-)=$}$}; 

%

\node (a1) at (0,0) {};
\node (a2) at (1,0) {};
\node (a3) at (2,0) {};
\node (a4) at (3,0) {};
\node (a5) at (4,0) {};

\node (b1) at (0,1) {};
\node (b2) at (1,1) {};
\node (b3) at (2,1) {};
\node (b4) at (3,2) {};
\node (b5) at (4,3) {};

\node (appo) at (0,3) {};

\draw (a1.center) .. controls (a1.4 north west) and (b1.4 south west) .. (b1.center);
\draw (a1) .. controls (a1.4 north east) and (b1.4 south east) .. (b1);

\draw (a2.center) .. controls (a2.4 north west) and (b2.4 south west) .. (b2.center);
\draw (a2) .. controls (a2.4 north east) and (b2.4 south east) .. (b2);

\draw (b1.center) .. controls (b1.8 north east) and (b4.8 north west) .. (b4);
\draw (a1.center) .. controls (a1.4 south east) and (a2.4 south west) .. (a2);

\draw (b1) .. controls (b1.16 north west) and (b5.4 south west) .. (b5.center);
\draw (b2) .. controls (b2.8 north west) and (b4.4 south west) .. (b4.center);

\draw (b2.center) .. controls (b2.4 north east) and (b3.4 north west) .. (b3);
\draw (a2.center) .. controls (a2.4 south east) and (a3.4 south west) .. (a3);
\draw (a3) .. controls (a3.4 north east) and (b3.4 south east) .. (b3);
\draw (a3.center) .. controls (a3.4 north west) and (b3.4 south west) .. (b3.center);
\draw (a1) .. controls (a1.16 south west) and (appo.16 west) .. (appo.center);
\draw (appo.center) .. controls (appo.4 east) and (b5.4 north west) .. (b5);

\draw (a3.center) .. controls (a3.4 south east) and (a4.4 south west) .. (a4);
\draw (a4.center) .. controls (a4.4 south east) and (a5.4 south west) .. (a5);
\draw (a4) .. controls (a4.4 north east) and (b4.4 south east) .. (b4);
\draw (b4.center) .. controls (b4.8 north east) and (a5.4 north west) .. (a5.center);
\draw (a5) .. controls (a5.4 north east) and (b5.4 south east) .. (b5);

\draw (a5.center) .. controls (a5.16 south east) and (b5.16 north east) .. (b5.center);
\draw (b3.center) .. controls (b3.4 north east) and (a4.4 north west) .. (a4.center);

\end{tikzpicture}
\]
\[
\begin{tikzpicture}[x=.6cm, y=.6cm, every path/.style={
 thick}, 
every node/.style={transform shape, knot crossing, inner sep=1.5pt}]

\node (aaaa) at (-4.75,1) {$\scalebox{1}{$(b)\;  \vec{\mathcal{L}}(T_+,T_-)=$}$}; 

%

\node (a1) at (0,0) {};
\node (a2) at (1,0) {};
\node (a3) at (2,0) {};
\node (a4) at (3,0) {};
\node (a5) at (4,0) {};

\node (b1) at (0,1) {};
\node (b2) at (1,1) {};
\node (b3) at (2,1) {};
\node (b4) at (3,2) {};
\node (b5) at (4,3) {};

\node (appo) at (0,3) {};

\draw[->] (a1.center) .. controls (a1.4 north west) and (b1.4 south west) .. (b1.center);
\draw[->]  (a1) .. controls (a1.4 north east) and (b1.4 south east) .. (b1);

\draw[<-]  (a2.center) .. controls (a2.4 north west) and (b2.4 south west) .. (b2.center);
\draw[<-]  (a2) .. controls (a2.4 north east) and (b2.4 south east) .. (b2);

\draw[->]  (b1.center) .. controls (b1.8 north east) and (b4.8 north west) .. (b4);
\draw[<-]  (a1.center) .. controls (a1.4 south east) and (a2.4 south west) .. (a2);

\draw[->]  (b1) .. controls (b1.16 north west) and (b5.4 south west) .. (b5.center);
\draw[<-]  (b2) .. controls (b2.8 north west) and (b4.4 south west) .. (b4.center);

\draw (b2.center) .. controls (b2.4 north east) and (b3.4 north west) .. (b3);
\draw[->]  (a2.center) .. controls (a2.4 south east) and (a3.4 south west) .. (a3);
\draw[->]  (a3) .. controls (a3.4 north east) and (b3.4 south east) .. (b3);
\draw[->]  (a3.center) .. controls (a3.4 north west) and (b3.4 south west) .. (b3.center);
\draw[<-]  (a1) .. controls (a1.16 south west) and (appo.16 west) .. (appo.center);
\draw[<-] (appo.center) .. controls (appo.4 east) and (b5.4 north west) .. (b5);

\draw[<-] (a3.center) .. controls (a3.4 south east) and (a4.4 south west) .. (a4);
\draw[->] (a4.center) .. controls (a4.4 south east) and (a5.4 south west) .. (a5);
\draw (a4) .. controls (a4.4 north east) and (b4.4 south east) .. (b4);
\draw[<-] (b4.center) .. controls (b4.8 north east) and (a5.4 north west) .. (a5.center);
\draw[->] (a5) .. controls (a5.4 north east) and (b5.4 south east) .. (b5);

\draw[<-] (a5.center) .. controls (a5.16 south east) and (b5.16 north east) .. (b5.center);
\draw[->] (b3.center) .. controls (b3.4 north east) and (a4.4 north west) .. (a4.center);

\node (x1) at (-.5,.5) {$\scalebox{0.75}{$+$}$};
\node (x2) at (.5,.5) {$\scalebox{0.75}{$-$}$};
\node (x3) at (1.45,.5) {$\scalebox{0.75}{$+$}$};
\node (x4) at (2.5,.5) {$\scalebox{0.75}{$-$}$};
\node (x5) at (3.4,.5) {$\scalebox{0.75}{$+$}$};
\node (x6) at (4.4,.5) {$\scalebox{0.75}{$-$}$};

\end{tikzpicture}
\]
%
\end{figure}
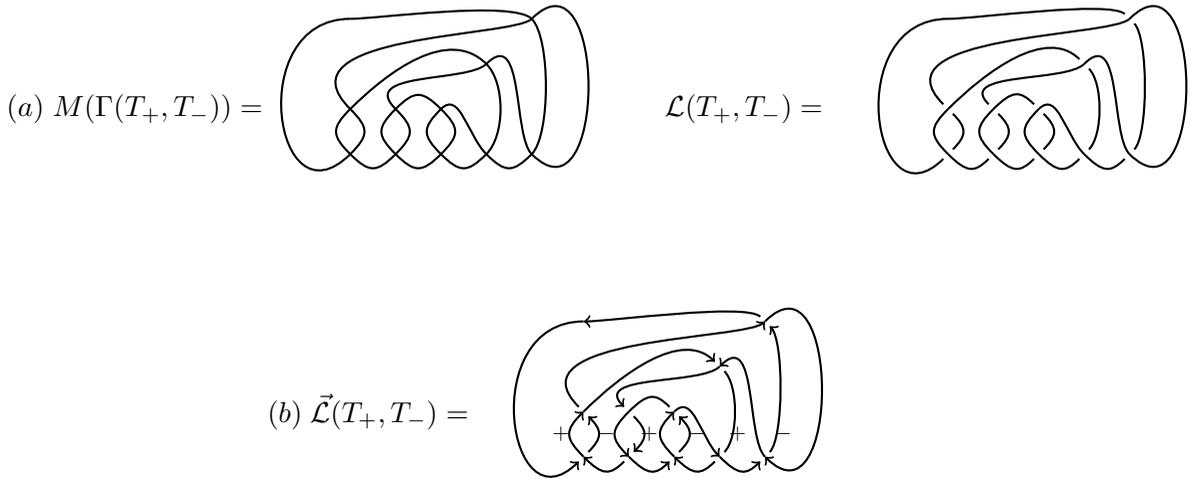

\section{Main result}
Before stating the main result of this paper, we recall the definition of the semi-dual graph $\Gamma(\vec{L})$ associated with an oriented link $\vec{L}$, \cite[Definition 5.3.2]{Jo14}.
Given an oriented link diagram, we colour the regions of the diagram in black and white (the unbounded region is coloured in white). 
The vertices of the semi-dual graph $\Gamma(\vec{L})$ correspond to the black regions of the link diagram, while its edges  correspond to the crossings and come with a sign according to the  rule shown in Figure \ref{fig5}. \\

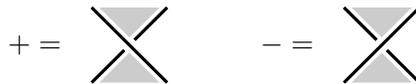
\begin{figure}
\caption{Signs of crossings. }\label{fig5}
\phantom{This text will be invisible} 
\[\begin{tikzpicture}[every path/.style={very thick}, every node/.style={transform shape, knot crossing, inner sep=1.5pt}]

\node (aaaa) at (-.75,0.5) {$\scalebox{1}{$+=$}$}; 

\node (a1) at (0,0) {}; 
\node (a2) at (0,1) {};
\node (a3) at (1,1) {};
\node (a4) at (1,0) {};
\node (a5) at (0.5,0.5) {};
%
%
\draw (a1.center) .. controls (a1.4 north east) .. (a5) .. controls (a5.4 north east) .. (a3.center);
\draw (a2.center) .. controls (a2.4 south east) .. (a5.center) .. controls (a5.4 south east) .. (a4.center);

  \fill [color=black,opacity=0.2]
               (0.1,0) -- (.5,.4) -- (.9,0) --(.1, 0);

  \fill [color=black,opacity=0.2]
               (0.1,1) -- (.5,.6) -- (.9,1) --(.1, 1); 
\end{tikzpicture}
\qquad 
\qquad\qquad
\begin{tikzpicture}[every path/.style={very thick}, every node/.style={transform shape, knot crossing, inner sep=1.5pt}]

\node (aaaa) at (1.5,0.5) {$\scalebox{1}{$-=$}$}; 

\node (b1) at (2.25,0) {}; 
\node (b2) at (2.25,1) {};
\node (b3) at (3.25,1) {};
\node (b4) at (3.25,0) {};
\node (b5) at (2.75,0.5) {};
%
%
\draw (b1.center) .. controls (b1.4 north east) .. (b5.center) .. controls (b5.4 north east) .. (b3.center);
\draw (b2.center) .. controls (b2.4 south east) .. (b5) .. controls (b5.4 south east) .. (b4.center);
%
%
%

  \fill [color=black,opacity=0.2]
               (2.35,0) -- (2.75,.4) -- (3.15,0) --(2.35, 0); 

  \fill [color=black,opacity=0.2]
               (2.35,1) -- (2.75,.6) -- (3.15,1) --(2.35, 1);

\end{tikzpicture}\]
\end{figure}
Moreover, if a shaded region has boundary positively oriented, we assign $+$ to the corresponding vertex of $\Gamma(\vec{L})$. Similarly, we assign $-$ to the vertices corresponding to the shaded regions whose boundary is negatively oriented. 
Semi-dual graphs are as good as link diagrams in representing knots and links. In this framework one can describe the three Reidemeister moves. In this paper we don't need all the moves, therefore we  recall in Figure \ref{Reide} only
the three  local moves relevant to the proof of the main result 
(if an edge has no sign, it means that the sign is arbitrary). 
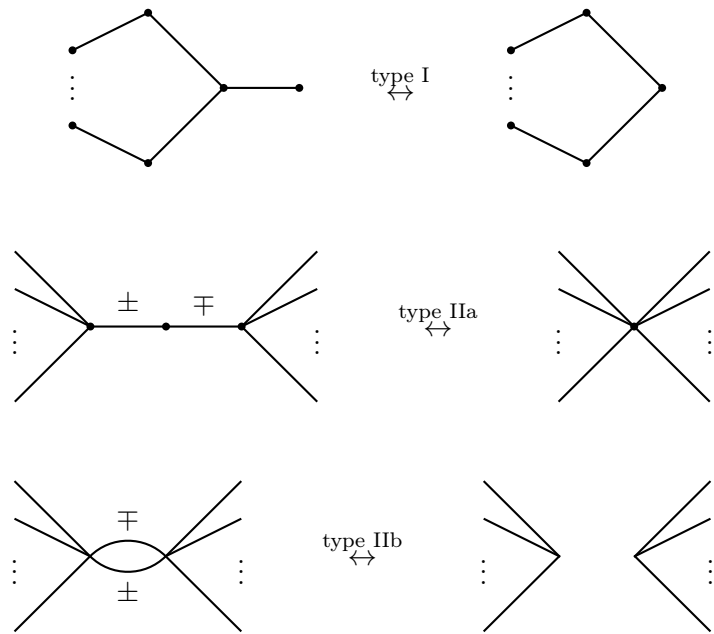
\begin{figure}
\caption{
The Reidemeister moves relevant to this paper.
A move of type I allows us to add (or remove) a 
1-valent vertex and its edge. 
In the presence of a 2-valent vertex whose edges have opposite signs, the two edges may be contracted according to a move of type IIa.
In a move of type IIb a pair of parallel edges with  opposite signs may be added (or removed). 
}
\label{Reide}
\[\begin{tikzpicture}[x=2cm, y=2cm,
    every edge/.style={
        draw,
      postaction={decorate,
                    decoration={markings}
                   }
        }
]

\node at (0,0.05) {$
\vdots$}; 

\draw[thick, thick] (0,0.25)--(.5,.5)--(1,0)--(.5,-.5)--(0,-0.25);
\draw[thick, thick] (1,0)--(1.5,0);
\fill (0,0.25)  circle[radius=1.5pt];
\fill (0,-0.25)  circle[radius=1.5pt];
\fill (.5,.5)  circle[radius=1.5pt];
\fill (1,0)  circle[radius=1.5pt];
\fill (.5,-.5)  circle[radius=1.5pt];
\fill (1.5,0)  circle[radius=1.5pt];

\end{tikzpicture}
\qquad 
\begin{tikzpicture}[x=2cm, y=2cm,
    every edge/.style={
        draw,
      postaction={decorate,
                    decoration={markings}
                   }
        }
]

\node at (0,.5) {
$\stackrel{\textrm{type I}}{\leftrightarrow}$}; 

\node at (0,0) {$\;$};

\end{tikzpicture}
\qquad
\begin{tikzpicture}[x=2cm, y=2cm,
    every edge/.style={
        draw,
      postaction={decorate,
                    decoration={markings}
                   }
        }
]

\node at (0,0.05) {
$\vdots$}; 

\draw[thick, thick] (0,0.25)--(.5,.5)--(1,0)--(.5,-.5)--(0,-0.25);
\fill (0,0.25)  circle[radius=1.5pt];
\fill (0,-0.25)  circle[radius=1.5pt];
\fill (.5,.5)  circle[radius=1.5pt];
\fill (1,0)  circle[radius=1.5pt];
\fill (.5,-.5)  circle[radius=1.5pt];

\end{tikzpicture}\]
\newline
\[
\begin{tikzpicture}[x=2cm, y=2cm,
    every edge/.style={
        draw,
      postaction={decorate,
                    decoration={markings}
                   }
        }
]

\node at (-0.5,-0.05) {
$\vdots$}; 
\node at (1.5,-0.05) {
$\vdots$}; 

\node at (.25,.15) {
$\pm$}; 
\node at (.75,.15) {
$\mp$}; 

\draw[thick] (0,0)--(-.5,.5);
\draw[thick] (0,0)--(-.5,-.5);
\draw[thick] (0,0)--(-.5,0.25);

\draw[thick] (0,0)--(.5,0);
\draw[thick] (.5,0)--(1,0);
\draw[thick] (1.5,0.25)--(1,0);
\draw[thick] (1.5,.5)--(1,0);
\draw[thick] (1.5,-.5)--(1,0);
\fill (0,0)  circle[radius=1.5pt];
\fill (.5,0)  circle[radius=1.5pt];
\fill (1,0)  circle[radius=1.5pt];

\end{tikzpicture}
\qquad 
\begin{tikzpicture}[x=2cm, y=2cm,
    every edge/.style={
        draw,
      postaction={decorate,
                    decoration={markings}
                   }
        }
]

\node at (0,.5) {
$\stackrel{\textrm{type IIa}}{\leftrightarrow}$}; 

\node at (0,0) {$\;$};

\end{tikzpicture}
\qquad
\begin{tikzpicture}[x=2cm, y=2cm,
    every edge/.style={
        draw,
      postaction={decorate,
                    decoration={markings}
                   }
        }
]

\node at (-0.5,-0.05) {
$\vdots$};
\node at (0.5,-0.05) {
$\vdots$}; 

\draw[thick] (0,0)--(-.5,.5);
\draw[thick] (0,0)--(-.5,-.5);
\draw[thick] (0,0)--(-.5,0.25);

\draw[thick] (.5,0.25)--(0,0);
\draw[thick] (.5,.5)--(0,0);
\draw[thick] (.5,-.5)--(0,0);
\fill (0,0)  circle[radius=1.5pt];

\end{tikzpicture}
\]
\newline
\[
\begin{tikzpicture}[x=2cm, y=2cm,
    every edge/.style={
        draw,
      postaction={decorate,
                    decoration={markings}
                   }
        }
]

\node at (.25,-.25) {
$\pm$}; 
\node at (.25,.25) {
$\mp$}; 

\draw[thick] (0,0)--(-.5,.5);
\draw[thick] (0,0)--(-.5,-.5);
\draw[thick] (0,0)--(-.5,0.25);

\draw[thick] (.5,0)--(1,0.25);
\draw[thick] (1,.5)--(.5,0);
\draw[thick] (1,-.5)--(.5,0);
\fill (0,0)  circle[radius=.5pt];
\fill (.5,0)  circle[radius=.5pt];

\draw[thick] (0,0) to[out=-45,in=-135] (.5,0);
\draw[thick] (0,0) to[out=45,in=135] (.5,0);

\node at (-0.5,-0.05) {
$\vdots$}; 
\node at (1,-0.05) {
$\vdots$}; 

\end{tikzpicture}
\qquad 
\begin{tikzpicture}[x=2cm, y=2cm,
    every edge/.style={
        draw,
      postaction={decorate,
                    decoration={markings}
                   }
        }
]

\node at (0,.5) {
$\stackrel{\textrm{type IIb}}{\leftrightarrow}$}; 

\node at (0,0) {$\;$};

\end{tikzpicture}
\qquad
\begin{tikzpicture}[x=2cm, y=2cm,
    every edge/.style={
        draw,
      postaction={decorate,
                    decoration={markings}
                   }
        }
]

\node at (-0.5,-0.05) {
$\vdots$}; 
\node at (1,-0.05) {
$\vdots$}; 

\draw[thick] (0,0)--(-.5,.5);
\draw[thick] (0,0)--(-.5,-.5);
\draw[thick] (0,0)--(-.5,0.25);

\draw[thick] (.5,0)--(1,0.25);
\draw[thick] (1,.5)--(.5,0);
\draw[thick] (1,-.5)--(.5,0);
\fill (0,0)  circle[radius=.5pt];
\fill (.5,0)  circle[radius=.5pt];


\end{tikzpicture}
\]
\end{figure}
We say that two signed graphs $\Gamma_1$ and $\Gamma_2$ are $2$-equivalent if they differ by planar isotopies and any of these three moves. 
Clearly, if two signed graphs are $2$-equivalent, then the associated knots/links are the same because, in terms of knot diagrams, the three moves are just Reidemeister moves of type I and II. 

We recall from \cite[Section 5.3.1, p. 34]{Jo14} that given an element $(T_+,T_-)$ of the Thompson group,
the semi-dual graph of $\mathcal{L}(T_+,T_-)$ coincides with 
the graph $\Gamma(T_+,T_-)$ described in the previous section.

Here follows a simple lemma which will come in handy in the proof of the main result of this paper.
\begin{lemma*} \label{Lemma1}
The following signed graphs are $2$-equivalent  
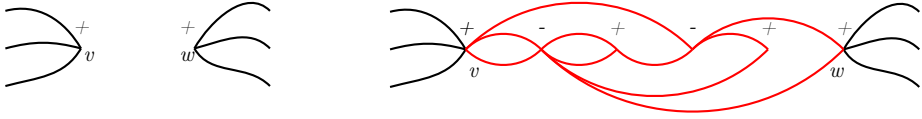
\[\begin{tikzpicture}[x=2cm, y=2cm,
    every edge/.style={
        draw,
      postaction={decorate,
                    decoration={markings}
                   }
        }
]

\draw[thick]  (-.5,0) to[out=15,in=165] (0,0);
\draw[thick]  (-.5,.25) to[out=15,in=120] (0,0);
\draw[thick]  (-.5,-.25) to[out=15,in=-120] (0,0);

\node at (0,.14) {$\scalebox{0.75}{+}$};
\node at (.05,-.05) {$\scalebox{0.75}{v}$};

\node at (.05,-.5) {$\scalebox{0.75}{$\;$}$};

\draw[thick]  (1.25,0) to[out=135,in=15] (.75,0);
\draw[thick]  (1.25,.25) to[out=135,in=60] (.75,0);
\draw[thick]  (1.25,-.25) to[out=135,in=-60] (.75,0);

\node at (.7,.14) {$\scalebox{0.75}{+}$};
\node at (.7,-.05) {$\scalebox{0.75}{w}$};

\end{tikzpicture}
\qquad \qquad
\begin{tikzpicture}[x=2cm, y=2cm,
    every edge/.style={
        draw,
      postaction={decorate,
                    decoration={markings}
                   }
        }
]

\draw[thick]  (-.5,0) to[out=15,in=165] (0,0);
\draw[thick]  (-.5,.25) to[out=15,in=120] (0,0);
\draw[thick]  (-.5,-.25) to[out=15,in=-120] (0,0);

\node at (0,.14) {$\scalebox{0.75}{+}$};
\node at (.05,-.15) {$\scalebox{0.75}{v}$};

\draw[thick]  (3,0) to[out=135,in=15] (2.5,0);
\draw[thick]  (3,.25) to[out=135,in=60] (2.5,0);
\draw[thick]  (3,-.25) to[out=135,in=-60] (2.5,0);

\node at (2.45,-.15) {$\scalebox{0.75}{w}$};


\draw[thick, red] (0,0) to[out=45,in=135] (.5,0);
\draw[thick, red] (.5,0) to[out=45,in=135] (1,0);
\draw[thick, red] (0,0) to[out=45,in=135] (1.5,0);
\draw[thick, red] (1.5,0) to[out=45,in=135] (2,0);
\draw[thick, red] (1.5,0) to[out=45,in=135] (2.5,0);

\draw[thick, red] (0,0) to[out=-45,in=-135] (.5,0);
\draw[thick, red] (.5,0) to[out=-45,in=-135] (1,0);
\draw[thick, red] (.5,0) to[out=-45,in=-135] (2,0);
\draw[thick, red] (.5,0) to[out=-45,in=-135] (2.5,0);
\draw[thick, red] (1,0) to[out=-45,in=-135] (1.5,0);

\node at (0,.14) {$\scalebox{0.75}{+}$};
\node at (.5,.14) {$\scalebox{0.75}{-}$};
\node at (1,.14) {$\scalebox{0.75}{+}$};
\node at (1.5,.14) {$\scalebox{0.75}{-}$};
\node at (2,.14) {$\scalebox{0.75}{+}$};
\node at (2.5,.14) {$\scalebox{0.75}{+}$};

\end{tikzpicture}
\]
where the edges of the red sub-graph
are positive if they lie in upper-half plane and negative if they are in the  lower-half plane (all the other edges are arbitrary).
\end{lemma*}
\begin{proof}
It's possible to recover the graph on the left 
 from the other one by applying six moves.
First of all, one erases the two pairs of   parallel edges  (type IIb moves) and contracts the two consecutive edges with opposite signs (type IIa move). Then, one applies a type IIb move to erase the new pair of parallel edges and, finally, 
the  graph on the left is recovered thanks to two  type I moves. 
\end{proof}

We are now 
in a position to state 
the main result of this paper. 
\begin{theorem*}
Given an oriented link $\vec L$ there is an element $g\in \vec{F}$ such that $\vec{\mathcal{L}}(g)=\vec{L}$.  
\end{theorem*}


The idea of the proof is fairly simple and we actually follow the same strategy used in \cite[Lemma 5.3.13]{Jo14}. Given an oriented link $\vec{L}$ we draw the semi-dual graph $\Gamma(\vec{L})$. 
If 
 $\Gamma(\vec{L})$ 
 consists of two trees, $\Gamma_+$ in the upper-half plane and $\Gamma_-$ in the lower-half plane,
 satisfying the following properties 
\cite[Proposition 4.1.3]{Jo14} 
\begin{enumerate}
\item the vertices are $(0,0)$, \ldots , $(N,0)$; 
\item each vertex other than $(0,0)$  is connected to exactly one vertex to its left;
\item each edge $e$ can be parametrized by a function $(x_e(t),y_e(t))$ with $x'_e(t)>0$, for all $t\in [0,1]$, and either $y_e(t)>0$, for all $t\in ]0,1[$ or $y_e(t)<0$, for all $t\in ]0,1[$;
\end{enumerate}
then we know that it is the $\Gamma$-graph of a certain element of $\vec{F}$ 
\cite[Lemma 4.1.4]{Jo14}. If it doesn't,  
we   apply the Reidemeister moves (actually the proof only requires the Reidemeister moves I and II) and get a graph which represents the same oriented knot/link, but satisfies the aforementioned conditions. 
In order to give the proof, we recall the notation from Jones' paper \cite[Definition 5.3.7 and Definition 5.3.9]{Jo14}.
Given a vertex $v$ of $\Gamma(\vec{L})$ we set 
\begin{align*}
e^{\rm up}& :=\{e\in E(\Gamma(\vec{L}))\; | \; e \textrm{ lies in the upper-half plane}\}\\
e^{\rm down}& :=\{e\in E(\Gamma(\vec{L}))\; | \; e \textrm{ lies in the lower-half plane}\}\\
e_v^{\rm in}& :=\{e\in E(\Gamma(\vec{L}))\; | \; \textrm{target}(e)=v \}\\
e_v^{\rm out}& :=\{e\in E(\Gamma(\vec{L}))\; | \; \textrm{source}(e)=v \}\\
\end{align*}

Here follows the proof of the theorem.
\begin{proof}
Let $\vec{L}$ and $\Gamma(\vec{L})$ be an oriented knot/link and it semi-dual graph, respectively.
Thanks to \cite[Lemma 5.3.6]{Jo14} we may suppose that $\Gamma(\vec{L})$ is \emph{standard} in the sense of  \cite[Definition 5.3.5]{Jo14}. 
This means that vertices are $(0,0)$, \ldots , $(N,0)$ and each edge $e$ can be parametrized by a function $(x_e(t),y_e(t))$ with $x'_e(t)>0$, for all $t\in [0,1]$, and either $y_e(t)>0$, for all $t\in ]0,1[$ or $y_e(t)<0$, for all $t\in ]0,1[$.
There are some cases that we have to consider now.

\noindent
\textbf{Case 1}: There is a vertex $w$ different from $(0,0)$ with $e^{\rm in}_w=\emptyset$. 
Denote by $v$ the vertex to the left of $w$.
We have to consider two subcases depending on whether $v$ and $w$ have the same colour or not, which (up to symmetry) are

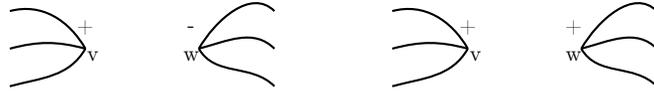
\[\begin{tikzpicture}[x=2cm, y=2cm,
    every edge/.style={
        draw,
      postaction={decorate,
                    decoration={markings}
                   }
        }
]

\draw[thick] (-.5,0) to[out=15,in=165] (0,0);
\draw[thick] (-.5,.25) to[out=15,in=120] (0,0);
\draw[thick] (-.5,-.25) to[out=15,in=-120] (0,0);

\node at (0,.14) {$\scalebox{0.75}{+}$};
\node at (.05,-.05) {$\scalebox{0.75}{v}$};

\draw[thick] (1.25,0) to[out=135,in=15] (.75,0);
\draw[thick] (1.25,.25) to[out=135,in=60] (.75,0);
\draw[thick] (1.25,-.25) to[out=135,in=-60] (.75,0);

\node at (.7,.14) {$\scalebox{0.75}{-}$};
\node at (.7,-.05) {$\scalebox{0.75}{w}$};

\end{tikzpicture}
\qquad 
\qquad 
\begin{tikzpicture}[x=2cm, y=2cm,
    every edge/.style={
        draw,
      postaction={decorate,
                    decoration={markings}
                   }
        }
]

\draw[thick] (-.5,0) to[out=15,in=165] (0,0);
\draw[thick] (-.5,.25) to[out=15,in=120] (0,0);
\draw[thick] (-.5,-.25) to[out=15,in=-120] (0,0);

\node at (0,.14) {$\scalebox{0.75}{+}$};
\node at (.05,-.05) {$\scalebox{0.75}{v}$};

\draw[thick] (1.25,0) to[out=135,in=15] (.75,0);
\draw[thick] (1.25,.25) to[out=135,in=60] (.75,0);
\draw[thick] (1.25,-.25) to[out=135,in=-60] (.75,0);

\node at (.7,.14) {$\scalebox{0.75}{+}$};
\node at (.7,-.05) {$\scalebox{0.75}{w}$};

\end{tikzpicture}
\]
The first subcase (i.e. the one in which the colour of $v$ and $w$ are different) can be handled as in \cite[proof of Lemma 5.3.13, Case 1]{Jo14}, while for the second one we need a new substitution. Here follow the substitutions for both the subcases.
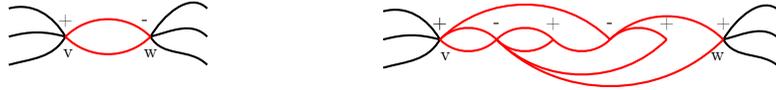
\[\begin{tikzpicture}[x=1.5cm, y=1.5cm,
    every edge/.style={
        draw,
      postaction={decorate,
                    decoration={markings}
                   }
        }
]

\draw[thick] (-.5,0) to[out=15,in=165] (0,0);
\draw[thick] (-.5,.25) to[out=15,in=120] (0,0);
\draw[thick] (-.5,-.25) to[out=15,in=-120] (0,0);

\node at (0,.14) {$\scalebox{0.65}{+}$};
\node at (.025,-.15) {$\scalebox{0.65}{v}$};

\draw[thick] (1.25,0) to[out=135,in=15] (.75,0);
\draw[thick] (1.25,.25) to[out=135,in=60] (.75,0);
\draw[thick] (1.25,-.25) to[out=135,in=-60] (.75,0);

\node at (.7,.14) {$\scalebox{0.65}{-}$};
\node at (.75,-.15) {$\scalebox{0.65}{w}$};

\node at (.75,-.5) {$\scalebox{0.65}{$\;$}$};

\draw[thick, red] (.,0) to[out=-45,in=-135] (.75,0);
\draw[thick, red] (.,0) to[out=45,in=135] (.75,0);

\end{tikzpicture}
\qquad\qquad\qquad 
\begin{tikzpicture}[x=1.5cm, y=1.5cm,
    every edge/.style={
        draw,
      postaction={decorate,
                    decoration={markings}
                   }
        }
]

\draw[thick] (-.5,0) to[out=15,in=165] (0,0);
\draw[thick] (-.5,.25) to[out=15,in=120] (0,0);
\draw[thick] (-.5,-.25) to[out=15,in=-120] (0,0);

\node at (0,.14) {$\scalebox{0.65}{+}$};
\node at (.05,-.15) {$\scalebox{0.65}{v}$};

\draw[thick] (3,0) to[out=135,in=15] (2.5,0);
\draw[thick] (3,.25) to[out=135,in=60] (2.5,0);
\draw[thick] (3,-.25) to[out=135,in=-60] (2.5,0);

\node at (2.45,-.15) {$\scalebox{0.65}{w}$};


\draw[thick, red] (0,0) to[out=45,in=135] (.5,0);
\draw[thick, red] (.5,0) to[out=45,in=135] (1,0);
\draw[thick, red] (0,0) to[out=45,in=135] (1.5,0);
\draw[thick, red] (1.5,0) to[out=45,in=135] (2,0);
\draw[thick, red] (1.5,0) to[out=45,in=135] (2.5,0);

\draw[thick, red] (0,0) to[out=-45,in=-135] (.5,0);
\draw[thick, red] (.5,0) to[out=-45,in=-135] (1,0);
\draw[thick, red] (.5,0) to[out=-45,in=-135] (2,0);
\draw[thick, red] (.5,0) to[out=-45,in=-135] (2.5,0);
\draw[thick, red] (1,0) to[out=-45,in=-135] (1.5,0);

\node at (0,.14) {$\scalebox{0.65}{+}$};
\node at (.5,.14) {$\scalebox{0.65}{-}$};
\node at (1,.14) {$\scalebox{0.65}{+}$};
\node at (1.5,.14) {$\scalebox{0.65}{-}$};
\node at (2,.14) {$\scalebox{0.65}{+}$};
\node at (2.5,.14) {$\scalebox{0.65}{+}$};

\end{tikzpicture}
\]
(When we insert new edges, those in upper-half plane are positive, whereas those in the lower-half plane are negative.)
The  graph on the left is clearly $2$-equivalent to the original graph. Indeed, it is enough to apply a type IIb move. 
As for graph on the right, the claim is precisely the content of the Lemma. 

\noindent
\textbf{Case 2}: $\Gamma(\vec{L})$ has a vertex $w$ with $|e_w^{\rm in}|=1$. We discuss the case in which the incoming edge to $w$ is in the upper-half plane, the other case can be handled in a similar manner.
As in the former case we   have two possible situations
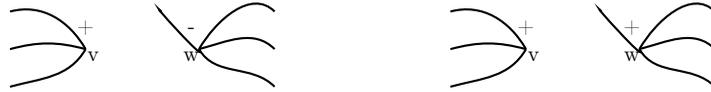
\[\begin{tikzpicture}[x=2cm, y=2cm,
    every edge/.style={
        draw,
      postaction={decorate,
                    decoration={markings}
                   }
        }
]

\draw[thick] (-.5,0) to[out=15,in=165] (0,0);
\draw[thick] (-.5,.25) to[out=15,in=120] (0,0);
\draw[thick] (-.5,-.25) to[out=15,in=-120] (0,0);

\node at (0,.14) {$\scalebox{0.75}{+}$};
\node at (.05,-.05) {$\scalebox{0.75}{v}$};

\draw[thick] (1.25,0) to[out=135,in=15] (.75,0);
\draw[thick] (1.25,.25) to[out=135,in=60] (.75,0);
\draw[thick] (1.25,-.25) to[out=135,in=-60] (.75,0);
\draw[thick] (.5,.25) to[out=135,in=-60] (.75,0);

\node at (.7,.14) {$\scalebox{0.75}{-}$};
\node at (.7,-.05) {$\scalebox{0.75}{w}$};

\end{tikzpicture}
\qquad\qquad\qquad 
\begin{tikzpicture}[x=2cm, y=2cm,
    every edge/.style={
        draw,
      postaction={decorate,
                    decoration={markings}
                   }
        }
]

\draw[thick] (-.5,0) to[out=15,in=165] (0,0);
\draw[thick] (-.5,.25) to[out=15,in=120] (0,0);
\draw[thick] (-.5,-.25) to[out=15,in=-120] (0,0);

\node at (0,.14) {$\scalebox{0.75}{+}$};
\node at (.05,-.05) {$\scalebox{0.75}{v}$};

\draw[thick] (1.25,0) to[out=135,in=15] (.75,0);
\draw[thick] (1.25,.25) to[out=135,in=60] (.75,0);
\draw[thick] (1.25,-.25) to[out=135,in=-60] (.75,0);
\draw[thick] (.5,.25) to[out=135,in=-60] (.75,0);

\node at (.7,.14) {$\scalebox{0.75}{+}$};
\node at (.7,-.05) {$\scalebox{0.75}{w}$};

\end{tikzpicture}
\]
These two cases can be dealt with the following substitutions.
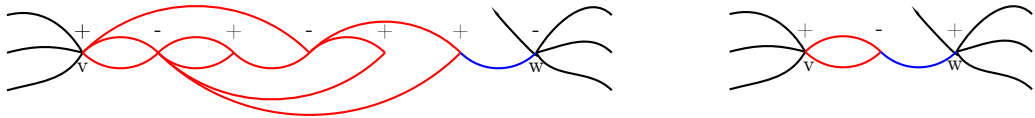
\[
\begin{tikzpicture}[x=2cm, y=2cm,
    every edge/.style={
        draw,
      postaction={decorate,
                    decoration={markings}
                   }
        }
]

\draw[thick] (-.5,0) to[out=15,in=165] (0,0);
\draw[thick] (-.5,.25) to[out=15,in=120] (0,0);
\draw[thick] (-.5,-.25) to[out=15,in=-120] (0,0);

\node at (0,.14) {$\scalebox{0.75}{+}$};
\node at (.0,-.1) {$\scalebox{0.75}{v}$};

\draw[thick] (3.5,0) to[out=135,in=15] (3,0);
\draw[thick] (3.5,.25) to[out=135,in=60] (3,0);
\draw[thick] (3.5,-.25) to[out=135,in=-60] (3,0);
\draw[thick] (2.75,.25) to[out=135,in=-60] (3,0);

\node at (3,.14) {$\scalebox{0.75}{-}$};
\node at (3,-.1) {$\scalebox{0.75}{w}$};

\draw[thick, red] (0,0) to[out=45,in=135] (.5,0);
\draw[thick, red] (.5,0) to[out=45,in=135] (1,0);
\draw[thick, red] (0,0) to[out=45,in=135] (1.5,0);
\draw[thick, red] (1.5,0) to[out=45,in=135] (2,0);
\draw[thick, red] (1.5,0) to[out=45,in=135] (2.5,0);

\draw[thick, red] (0,0) to[out=-45,in=-135] (.5,0);
\draw[thick, red] (.5,0) to[out=-45,in=-135] (1,0);
\draw[thick, red] (.5,0) to[out=-45,in=-135] (2,0);
\draw[thick, red] (.5,0) to[out=-45,in=-135] (2.5,0);
\draw[thick, red] (1,0) to[out=-45,in=-135] (1.5,0);
\draw[thick, blue] (2.5,0) to[out=-45,in=-135] (3,0);

\node at (0,.14) {$\scalebox{0.75}{+}$};
\node at (.5,.14) {$\scalebox{0.75}{-}$};
\node at (1,.14) {$\scalebox{0.75}{+}$};
\node at (1.5,.14) {$\scalebox{0.75}{-}$};
\node at (2,.14) {$\scalebox{0.75}{+}$};
\node at (2.5,.14) {$\scalebox{0.75}{+}$};

\end{tikzpicture}
\qquad
\qquad
\begin{tikzpicture}[x=2cm, y=2cm,
    every edge/.style={
        draw,
      postaction={decorate,
                    decoration={markings}
                   }
        }
]

\draw[thick] (-.5,0) to[out=15,in=165] (0,0);
\draw[thick] (-.5,.25) to[out=15,in=120] (0,0);
\draw[thick] (-.5,-.25) to[out=15,in=-120] (0,0);

\node at (0,.14) {$\scalebox{0.75}{+}$};
\node at (.025,-.1) {$\scalebox{0.75}{v}$};

\node at (.025,-.5) {$\scalebox{0.75}{$\;$}$};

\draw[thick] (1.5,0) to[out=135,in=15] (1,0);
\draw[thick] (1.5,.25) to[out=135,in=60] (1,0);
\draw[thick] (1.5,-.25) to[out=135,in=-60] (1,0);
\draw[thick] (.75,.25) to[out=135,in=-60] (1,0);

\node at (.49,.14) {$\scalebox{0.75}{-}$};

\node at (.99,.14) {$\scalebox{0.75}{+}$};
\node at (.99,-.09) {$\scalebox{0.75}{w}$};


\draw[thick, red] (0,0) to[out=45,in=135] (.5,0);

\draw[thick, red] (0,0) to[out=-45,in=-135] (.5,0);
\draw[thick, blue] (.5,0) to[out=-45,in=-135] (1,0);

\end{tikzpicture}
\]
The  graph on the left is easily seen to be $2$-equivalent to the initial graph thanks to an application of the Lemma 
 (to the red subgraph) and a type I move.
The graph on the right can be dealt with the application of a type IIb move and a type I move.

\noindent
\textbf{Case 3}: The graph $\Gamma(\vec{L})$ has a vertex $v$ with $|e_v^{\rm in}\cap e^{\rm up}|>1$ or $|e_v^{\rm in}\cap e^{\rm down}|>1$.
We show how to deal with the first case, the other case being similar. Here is the situation and the corresponding substitution.

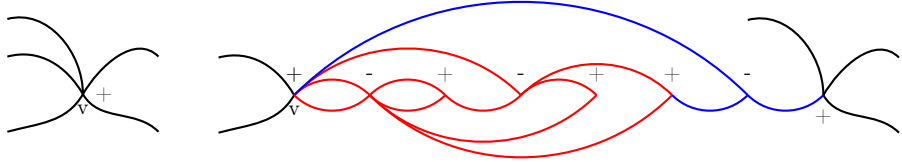
\[\begin{tikzpicture}[x=2cm, y=2cm,
    every edge/.style={
        draw,
      postaction={decorate,
                    decoration={markings}
                   }
        }
]

\draw[thick] (-.5,.5) to[out=15,in=90] (0,0);
\draw[thick] (-.5,.25) to[out=15,in=120] (0,0);
\draw[thick] (-.5,-.25) to[out=15,in=-120] (0,0);

\node at (.14,0) {$\scalebox{0.75}{+}$};
\node at (.0,-.1) {$\scalebox{0.75}{v}$};

\node at (-.5,-.5) {$\scalebox{0.75}{$\;$}$};

\draw[thick] (.5,.25) to[out=135,in=60] (.,0);
\draw[thick] (.5,-.25) to[out=135,in=-60] (.,0);

\end{tikzpicture}
\qquad
\begin{tikzpicture}[x=2cm, y=2cm,
    every edge/.style={
        draw,
      postaction={decorate,
                    decoration={markings}
                   }
        }
]

\draw[thick] (3,.5) to[out=15,in=90] (3.5,0);
\draw[thick] (-.5,.25) to[out=15,in=120] (0,0);
\draw[thick] (-.5,-.25) to[out=15,in=-120] (0,0);

\node at (0,.14) {$\scalebox{0.75}{+}$};
\node at (.0,-.1) {$\scalebox{0.75}{v}$};

\draw[thick] (4,.25) to[out=135,in=60] (3.5,0);
\draw[thick] (4,-.25) to[out=135,in=-60] (3.5,0);

\node at (3.5,-.14) {$\scalebox{0.75}{+}$};


\draw[thick, red] (0,0) to[out=45,in=135] (.5,0);
\draw[thick, red] (.5,0) to[out=45,in=135] (1,0);
\draw[thick, red] (0,0) to[out=45,in=135] (1.5,0);
\draw[thick, red] (1.5,0) to[out=45,in=135] (2,0);
\draw[thick, red] (1.5,0) to[out=45,in=135] (2.5,0);
\draw[thick, blue] (0,0) to[out=45,in=135] (3,0);

\draw[thick, red] (0,0) to[out=-45,in=-135] (.5,0);
\draw[thick, red] (.5,0) to[out=-45,in=-135] (1,0);
\draw[thick, red] (.5,0) to[out=-45,in=-135] (2,0);
\draw[thick, red] (.5,0) to[out=-45,in=-135] (2.5,0);
\draw[thick, red] (1,0) to[out=-45,in=-135] (1.5,0);
\draw[thick, blue] (2.5,0) to[out=-45,in=-135] (3,0);
\draw[thick, blue] (3,0) to[out=-45,in=-135] (3.5,0);

\node at (0,.14) {$\scalebox{0.75}{+}$};
\node at (.5,.14) {$\scalebox{0.75}{-}$};
\node at (1,.14) {$\scalebox{0.75}{+}$};
\node at (1.5,.14) {$\scalebox{0.75}{-}$};
\node at (2,.14) {$\scalebox{0.75}{+}$};
\node at (2.5,.14) {$\scalebox{0.75}{+}$};
\node at (3,.14) {$\scalebox{0.75}{-}$};

\end{tikzpicture}
\]
The  graph on the right can be seen to be $2$-equivalent to the other one by applying the  Lemma 
 (to the red subgraph), a type I move, and a type IIa move.

\noindent
\textbf{Case 4}: After repeated applications of the previous substitutions  all the vertices (except the leftmost) have two incoming edges. The only problem that we may have is a negative edge in the upper-half plane or a positive edge in lower-half plane, the other case is similar. We show how to deal with the first case. We have two subcases to deal with. It is appropriate to say a note of warning: in order to distinguish the colours of the vertices from the signs of the edges, we will colour the latters in red.
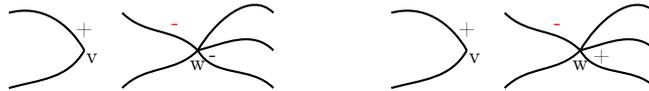
\[\begin{tikzpicture}[x=2cm, y=2cm,
    every edge/.style={
        draw,
      postaction={decorate,
                    decoration={markings}
                   }
        }
]

\draw[thick] (-.5,.25) to[out=15,in=120] (0,0);
\draw[thick] (-.5,-.25) to[out=15,in=-120] (0,0);

\node at (0,.14) {$\scalebox{0.75}{+}$};
\node at (.05,-.05) {$\scalebox{0.75}{v}$};

\draw[thick] (1.25,0) to[out=135,in=15] (.75,0);
\draw[thick] (1.25,.25) to[out=135,in=60] (.75,0);
\draw[thick] (1.25,-.25) to[out=135,in=-60] (.75,0);
\draw[thick] (.25,.25) to[out=-45,in=135] (.75,0);
\draw[thick] (.25,-.25) to[out=45,in=-135] (.75,0);

\node[red] at (.6,.17) {$\scalebox{0.75}{-}$};
\node at (.75,-.1) {$\scalebox{0.75}{w}$};
\node at (.85,-.04) {$\scalebox{0.75}{-}$};

\end{tikzpicture}
\qquad\qquad
\begin{tikzpicture}[x=2cm, y=2cm,
    every edge/.style={
        draw,
      postaction={decorate,
                    decoration={markings}
                   }
        }
]

\draw[thick] (-.5,.25) to[out=15,in=120] (0,0);
\draw[thick] (-.5,-.25) to[out=15,in=-120] (0,0);

\node at (0,.14) {$\scalebox{0.75}{+}$};
\node at (.05,-.05) {$\scalebox{0.75}{v}$};

\draw[thick] (1.25,0) to[out=135,in=15] (.75,0);
\draw[thick] (1.25,.25) to[out=135,in=60] (.75,0);
\draw[thick] (1.25,-.25) to[out=135,in=-60] (.75,0);
\draw[thick] (.25,.25) to[out=-45,in=135] (.75,0);
\draw[thick] (.25,-.25) to[out=45,in=-135] (.75,0);

\node[red] at (.6,.17) {$\scalebox{0.75}{-}$};
\node at (.75,-.1) {$\scalebox{0.75}{w}$};
\node at (.89,-.04) {$\scalebox{0.75}{+}$};

\end{tikzpicture}
\]
For the first subcase we use the following substitution.
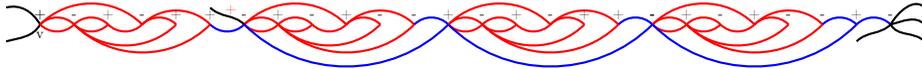
\[\begin{tikzpicture}[x=.9cm, y=.9cm,
    every edge/.style={
        draw,
      postaction={decorate,
                    decoration={markings}
                   }
        }
]

\draw[thick] (-.5,.25) to[out=15,in=120] (0,0);
\draw[thick] (-.5,-.25) to[out=15,in=-120] (0,0);

\node at (0,.14) {$\scalebox{0.5}{+}$};
\node at (.,-.15) {$\scalebox{0.5}{v}$};


\draw[thick, red] (0,0) to[out=45,in=135] (.5,0);
\draw[thick, red] (.5,0) to[out=45,in=135] (1,0);
\draw[thick, red] (0,0) to[out=45,in=135] (1.5,0);
\draw[thick, red] (1.5,0) to[out=45,in=135] (2,0);
\draw[thick, red] (1.5,0) to[out=45,in=135] (2.5,0);

\draw[thick, red] (0,0) to[out=-45,in=-135] (.5,0);
\draw[thick, red] (.5,0) to[out=-45,in=-135] (1,0);
\draw[thick, red] (.5,0) to[out=-45,in=-135] (2,0);
\draw[thick, red] (.5,0) to[out=-45,in=-135] (2.5,0);
\draw[thick, red] (1,0) to[out=-45,in=-135] (1.5,0);
\draw[thick, blue] (2.5,0) to[out=-45,in=-135] (3,0);

\node at (0,.14) {$\scalebox{0.5}{+}$};
\node at (.5,.14) {$\scalebox{0.5}{-}$};
\node at (1,.14) {$\scalebox{0.5}{+}$};
\node at (1.5,.14) {$\scalebox{0.5}{-}$};
\node at (2,.14) {$\scalebox{0.5}{+}$};
\node at (2.5,.14) {$\scalebox{0.5}{+}$};


\draw[thick, red] (3,0) to[out=45,in=135] (3.5,0);
\draw[thick, red] (3.5,0) to[out=45,in=135] (4,0);
\draw[thick, red] (3,0) to[out=45,in=135] (4.5,0);
\draw[thick, red] (4.5,0) to[out=45,in=135] (5,0);
\draw[thick, red] (4.5,0) to[out=45,in=135] (5.5,0);
\draw[thick, blue] (5.5,0) to[out=45,in=135] (6,0);

\draw[thick, red] (3,0) to[out=-45,in=-135] (3.5,0);
\draw[thick, red] (3.5,0) to[out=-45,in=-135] (4,0);
\draw[thick, red] (3.5,0) to[out=-45,in=-135] (5,0);
\draw[thick, red] (3.5,0) to[out=-45,in=-135] (5.5,0);
\draw[thick, red] (4,0) to[out=-45,in=-135] (4.5,0);
\draw[thick, blue] (3,0) to[out=-45,in=-135] (6,0);

\node at (3,.14) {$\scalebox{0.5}{-}$};
\node at (3.5,.14) {$\scalebox{0.5}{+}$};
\node at (4,.14) {$\scalebox{0.5}{-}$};
\node at (4.5,.14) {$\scalebox{0.5}{+}$};
\node at (5,.14) {$\scalebox{0.5}{-}$};
\node at (5.5,.14) {$\scalebox{0.5}{-}$};
\node at (6,.14) {$\scalebox{0.5}{+}$};


\draw[thick, red] (6,0) to[out=45,in=135] (6.5,0);
\draw[thick, red] (6.5,0) to[out=45,in=135] (7,0);
\draw[thick, red] (6,0) to[out=45,in=135] (7.5,0);
\draw[thick, red] (7.5,0) to[out=45,in=135] (8,0);
\draw[thick, red] (7.5,0) to[out=45,in=135] (8.5,0);
\draw[thick, blue] (8.5,0) to[out=45,in=135] (9,0);

\draw[thick, red] (6,0) to[out=-45,in=-135] (6.5,0);
\draw[thick, red] (6.5,0) to[out=-45,in=-135] (7,0);
\draw[thick, red] (6.5,0) to[out=-45,in=-135] (8,0);
\draw[thick, red] (6.5,0) to[out=-45,in=-135] (8.5,0);
\draw[thick, red] (7,0) to[out=-45,in=-135] (7.5,0);
\draw[thick, blue] (6,0) to[out=-45,in=-135] (9,0);

\node at (6.5,.14) {$\scalebox{0.5}{-}$};
\node at (7,.14) {$\scalebox{0.5}{+}$};
\node at (7.5,.14) {$\scalebox{0.5}{-}$};
\node at (8,.14) {$\scalebox{0.5}{+}$};
\node at (8.5,.14) {$\scalebox{0.5}{+}$};
\node at (9,.14) {$\scalebox{0.5}{-}$};


\draw[thick, red] (9,0) to[out=45,in=135] (9.5,0);
\draw[thick, red] (9.5,0) to[out=45,in=135] (10,0);
\draw[thick, red] (9,0) to[out=45,in=135] (10.5,0);
\draw[thick, red] (10.5,0) to[out=45,in=135] (11,0);
\draw[thick, red] (10.5,0) to[out=45,in=135] (11.5,0);
\draw[thick, blue] (11.5,0) to[out=45,in=135] (12,0);
\draw[thick, blue] (12,0) to[out=45,in=135] (12.5,0);

\draw[thick, red] (9,0) to[out=-45,in=-135] (9.5,0);
\draw[thick, red] (9.5,0) to[out=-45,in=-135] (10,0);
\draw[thick, red] (9.5,0) to[out=-45,in=-135] (11,0);
\draw[thick, red] (9.5,0) to[out=-45,in=-135] (11.5,0);
\draw[thick, red] (10,0) to[out=-45,in=-135] (10.5,0);
\draw[thick, blue] (9,0) to[out=-45,in=-135] (12,0);

\node at (9,.14) {$\scalebox{0.5}{-}$};
\node at (9.5,.14) {$\scalebox{0.5}{+}$};
\node at (10,.14) {$\scalebox{0.5}{-}$};
\node at (10.5,.14) {$\scalebox{0.5}{+}$};
\node at (11,.14) {$\scalebox{0.5}{-}$};
\node at (11.5,.14) {$\scalebox{0.5}{-}$};
\node at (12,.14) {$\scalebox{0.5}{+}$};
\node at (12.5,.14) {$\scalebox{0.5}{-}$};


\draw[thick] (13,0) to[out=135,in=15] (12.5,0);
\draw[thick] (13,.25) to[out=135,in=60] (12.5,0);
\draw[thick] (13,-.25) to[out=135,in=-60] (12.5,0);
\draw[thick] (2.5,.25) to[out=-45,in=135] (3,0);
\draw[thick] (12,-.25) to[out=45,in=-135] (12.5,0);

\node[red] at (2.8,.22) {$\scalebox{0.5}{+}$};

%
%
%


\end{tikzpicture}\]
whereas for the second subcase we use 
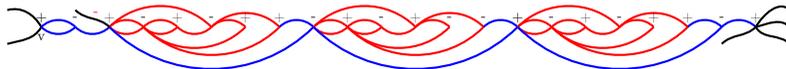
\[\begin{tikzpicture}[x=.9cm, y=.9cm,
    every edge/.style={
        draw,
      postaction={decorate,
                    decoration={markings}
                   }
        }
]

\draw[thick] (1.5,.25) to[out=15,in=120] (2,0);
\draw[thick] (1.5,-.25) to[out=15,in=-120] (2,0);

\node at (2,.14) {$\scalebox{0.5}{+}$};
\node at (2,-.15) {$\scalebox{0.5}{v}$};


\draw[thick, blue] (2,0) to[out=45,in=135] (2.5,0);

\draw[thick, blue] (2,0) to[out=-45,in=-135] (2.5,0);
\draw[thick, blue] (2.5,0) to[out=-45,in=-135] (3,0);

\node at (2,.14) {$\scalebox{0.5}{+}$};
\node at (2.5,.14) {$\scalebox{0.5}{-}$};


\draw[thick, red] (3,0) to[out=45,in=135] (3.5,0);
\draw[thick, red] (3.5,0) to[out=45,in=135] (4,0);
\draw[thick, red] (3,0) to[out=45,in=135] (4.5,0);
\draw[thick, red] (4.5,0) to[out=45,in=135] (5,0);
\draw[thick, red] (4.5,0) to[out=45,in=135] (5.5,0);
\draw[thick, blue] (5.5,0) to[out=45,in=135] (6,0);

\draw[thick, red] (3,0) to[out=-45,in=-135] (3.5,0);
\draw[thick, red] (3.5,0) to[out=-45,in=-135] (4,0);
\draw[thick, red] (3.5,0) to[out=-45,in=-135] (5,0);
\draw[thick, red] (3.5,0) to[out=-45,in=-135] (5.5,0);
\draw[thick, red] (4,0) to[out=-45,in=-135] (4.5,0);
\draw[thick, blue] (3,0) to[out=-45,in=-135] (6,0);

\node at (3,.14) {$\scalebox{0.5}{+}$};
\node at (3.5,.14) {$\scalebox{0.5}{-}$};
\node at (4,.14) {$\scalebox{0.5}{+}$};
\node at (4.5,.14) {$\scalebox{0.5}{-}$};
\node at (5,.14) {$\scalebox{0.5}{+}$};
\node at (5.5,.14) {$\scalebox{0.5}{+}$};
\node at (6,.14) {$\scalebox{0.5}{-}$};


\draw[thick, red] (6,0) to[out=45,in=135] (6.5,0);
\draw[thick, red] (6.5,0) to[out=45,in=135] (7,0);
\draw[thick, red] (6,0) to[out=45,in=135] (7.5,0);
\draw[thick, red] (7.5,0) to[out=45,in=135] (8,0);
\draw[thick, red] (7.5,0) to[out=45,in=135] (8.5,0);
\draw[thick, blue] (8.5,0) to[out=45,in=135] (9,0);

\draw[thick, red] (6,0) to[out=-45,in=-135] (6.5,0);
\draw[thick, red] (6.5,0) to[out=-45,in=-135] (7,0);
\draw[thick, red] (6.5,0) to[out=-45,in=-135] (8,0);
\draw[thick, red] (6.5,0) to[out=-45,in=-135] (8.5,0);
\draw[thick, red] (7,0) to[out=-45,in=-135] (7.5,0);
\draw[thick, blue] (6,0) to[out=-45,in=-135] (9,0);

\node at (6.5,.14) {$\scalebox{0.5}{+}$};
\node at (7,.14) {$\scalebox{0.5}{-}$};
\node at (7.5,.14) {$\scalebox{0.5}{+}$};
\node at (8,.14) {$\scalebox{0.5}{-}$};
\node at (8.5,.14) {$\scalebox{0.5}{-}$};
\node at (9,.14) {$\scalebox{0.5}{+}$};


\draw[thick, red] (9,0) to[out=45,in=135] (9.5,0);
\draw[thick, red] (9.5,0) to[out=45,in=135] (10,0);
\draw[thick, red] (9,0) to[out=45,in=135] (10.5,0);
\draw[thick, red] (10.5,0) to[out=45,in=135] (11,0);
\draw[thick, red] (10.5,0) to[out=45,in=135] (11.5,0);
\draw[thick, blue] (11.5,0) to[out=45,in=135] (12,0);
\draw[thick, blue] (12,0) to[out=45,in=135] (12.5,0);

\draw[thick, red] (9,0) to[out=-45,in=-135] (9.5,0);
\draw[thick, red] (9.5,0) to[out=-45,in=-135] (10,0);
\draw[thick, red] (9.5,0) to[out=-45,in=-135] (11,0);
\draw[thick, red] (9.5,0) to[out=-45,in=-135] (11.5,0);
\draw[thick, red] (10,0) to[out=-45,in=-135] (10.5,0);
\draw[thick, blue] (9,0) to[out=-45,in=-135] (12,0);

\node at (9,.14) {$\scalebox{0.5}{+}$};
\node at (9.5,.14) {$\scalebox{0.5}{-}$};
\node at (10,.14) {$\scalebox{0.5}{+}$};
\node at (10.5,.14) {$\scalebox{0.5}{-}$};
\node at (11,.14) {$\scalebox{0.5}{+}$};
\node at (11.5,.14) {$\scalebox{0.5}{+}$};
\node at (12,.14) {$\scalebox{0.5}{-}$};
\node at (12.5,.14) {$\scalebox{0.5}{+}$};


\draw[thick] (13,0) to[out=135,in=15] (12.5,0);
\draw[thick] (13,.25) to[out=135,in=60] (12.5,0);
\draw[thick] (13,-.25) to[out=135,in=-60] (12.5,0);
\draw[thick] (2.5,.25) to[out=-45,in=135] (3,0);
\draw[thick] (12,-.25) to[out=45,in=-135] (12.5,0);

\node[red] at (2.8,.22) {$\scalebox{0.5}{-}$};

%
%
%


\end{tikzpicture}
\]
The substitution of the first subcase does not affect the associated knot/link because one may apply the Lemma 
  four times to the red subgraphs, then four times a type I move and finally twice the type IIa move. 
For the other substitution one may erase a pair of parallel edges, 
 apply the Lemma 
  three times to the  red subgraphs, four times a type I move,   
and finally use twice the type IIa move.

\end{proof}
Now that we have an Alexander Theorem for $\vec F$ at our disposal,
to reinforce the analogy with the braid groups it would be important to obtain an analogue of Markov Theorem in this framework. 
We plan to return to this problem in the future.

\section*{Acknowledgements}
The author is grateful   to
Simone Del Vecchio for his valuable comments on the first draft of this paper 
and 
to the Swiss National Science Foundation for the support. 
Finally, 
the author would like to thank the referee 
for his or her particularly attentive perusal of the manuscript, which resulted in many improvements of the paper.

\end{document}